%% file: context-free2.tex
\documentclass[11pt]{article}
\usepackage[margin=2.2cm,centering,nohead]{geometry}
\usepackage{amsmath, amsthm, amssymb}
\usepackage{mathrsfs}
\usepackage{graphicx}
\usepackage[small,it]{caption}
\usepackage{tikz}
\usetikzlibrary{calc,backgrounds,patterns,decorations.pathreplacing}
\usepackage[T1]{fontenc}
\usepackage[sc]{mathpazo}
\usepackage{textcomp} 
\usepackage{parskip}
\usepackage{verbatim}
\usepackage{todonotes}

\definecolor{darkblue}{rgb}{0, 0, .4}
\usepackage[bookmarks]{hyperref}
\hypersetup{
    colorlinks=true,
    linkcolor=darkblue,
    anchorcolor=darkblue,
    citecolor=darkblue,
    urlcolor=darkblue,
    pdfpagemode=UseThumbs,
    pdftitle={Combinatorial specifications for juxtapositions of permutation classes},
    pdfsubject={Combinatorics},
    pdfauthor={R. Brignall and J. Sliacan},
    pdfkeywords={permutation,pattern,simple permutation,grid classes}
}

\title{Combinatorial specifications for juxtapositions of permutation classes}
\author{
Robert Brignall\thanks{Email addresses: \texttt{rbrignall@gmail.com, jakub.sliacan@gmail.com}.}\\[10pt]
\small School of Mathematics and Statistics\\
\small The Open University\\
\small Milton Keynes, MK7 6AA\\
\small United Kingdom
\and
Jakub Slia\v{c}an\footnotemark[1]\\[10pt]
\small Matematik och matematisk statistik\\
\small Ume\aa\ universitet\\
\small 901 87 Ume\aa\\
\small Sweden
}

\newcommand{\A}{\mathcal{A}}

\newcommand{\AR}{\A^{\mathsf{R}}}

\newcommand{\B}{\mathcal{B}}

\newcommand{\C}{\mathcal{C}}
\newcommand{\Cp}{\C^{\oplus}}
\newcommand{\Cm}{\C^{\ominus}}
\newcommand{\CL}{\C^{\mathsf{L}}}
\newcommand{\CR}{\C^{\mathsf{R}}}
\newcommand{\CLR}{\C^{\mathsf{LR}}}

\newcommand{\D}{\mathcal{D}}
\newcommand{\DD}{\ensuremath{\mathcal{D}}}

\newcommand{\EE}{\mathcal{E}}
\newcommand{\F}{\mathcal{F}}

\newcommand{\M}{\ensuremath{\mathcal{M}}}

\renewcommand{\S}{\mathcal{S}}

\newcommand{\SZ}{\S_{\Z}}
\newcommand{\SZp}{\S_{\Z}^+}

\newcommand{\V}{\ensuremath{\mathcal{V}}}
\newcommand{\Z}{\mathcal{Z}}
\newcommand{\ZL}{\Z^{\mathsf{L}}}
\newcommand{\ZR}{\Z^{\mathsf{R}}}
\newcommand{\ZLR}{\Z^{\mathsf{LR}}}
\newcommand{\bop}{\mathcal{O}}

\newcommand{\seq}{\textsc{Seq}}
\newcommand{\seqp}{\seq^+}

\newcommand{\Av}{\operatorname{Av}}
\newcommand{\Si}{\operatorname{Si}}

\newcommand{\rS}{\ensuremath{\mathscr{S}}}

\renewcommand{\o}{{\epsilon}}
\renewcommand{\i}{{\bullet}}
\newcommand{\oo}{{\circ\circ}}
\newcommand{\oi}{{\circ\bullet}}
\newcommand{\io}{{\bullet\circ}}
\newcommand{\ii}{{\bullet\bullet}}
\newcommand{\opo}{\Omega_\o}
\newcommand{\opi}{\Omega_\i}
\newcommand{\opoo}{\Omega_\oo}
\newcommand{\opio}{\Omega_\io}
\newcommand{\opoi}{\Omega_\oi}
\newcommand{\opii}{\Omega_\ii}

\newcommand{\cplusc}[2]{
  \begin{tikzpicture}[baseline=4ex, scale=0.6]
      \draw (0,0) rectangle (1,1) node[pos=0.5]{\ensuremath{#1}};
      \draw (1,1) rectangle (2,2) node[pos=0.5]{\ensuremath{#2}};
    \end{tikzpicture}}

 \newcommand{\cminusc}[2]{
  \begin{tikzpicture}[baseline=4ex, scale=0.6]
      \draw (1,0) rectangle (2,1) node[pos=0.5]{\ensuremath{#1}};
      \draw (0,1) rectangle (1,2) node[pos=0.5]{\ensuremath{#2}};
  \end{tikzpicture}}

\begingroup
    \makeatletter
    \@for\theoremstyle:=definition,remark,plain\do{%
        \expandafter\g@addto@macro\csname th@\theoremstyle\endcsname{%
            \addtolength\thm@preskip\parskip
            }%
        }
\endgroup

\theoremstyle{plain}
\newtheorem{theorem}{Theorem}[section]

\newtheorem{lemma}[theorem]{Lemma}

\newtheorem{corollary}[theorem]{Corollary}

\newtheorem{proposition}[theorem]{Proposition}
\newtheorem{observation}[theorem]{Observation}

\usepackage{paralist}

\input{tikzperm.tex}

\begin{document}
\maketitle
\begin{abstract}
We show that, given a suitable combinatorial specification for a permutation class $\mathcal{C}$, one can obtain a specification for the juxtaposition (on either side) of $\mathcal{C}$ with Av(21) or Av(12), and that if the enumeration for $\mathcal{C}$ is given by a rational or algebraic generating function, so is the enumeration for the juxtaposition. Furthermore this process can be iterated, thereby providing an effective method to enumerate any `skinny' $k\times 1$ grid class in which at most one cell is non-monotone, with a guarantee on the nature of the enumeration given the nature of the enumeration of the non-monotone cell.
\end{abstract}

%
%
%
%
%
%
%
%
\section{Introduction}\label{sec-intro}

A phenomenon observed both in the practical enumeration of specific permutation classes and in the study of growth rates is that in many classes, the entries of any permutation in the class can be partitioned in such a way as the pattern formed by the entries in each part lies in some specified proper subclass.

As well as very general compositions such as the \emph{merge} of two classes (exemplified by the class $\Av(321)$, being the merge of two increasing sequences), one more restrictive (but nonetheless ubiquitous) instance of this phenomenon is the \emph{juxtaposition} of two classes, first considered by Atkinson~\cite{atkinson:restricted-perm:}. For a recent comprehensive survey of the study of permutation patterns, and particularly notions of structure, see Vatter's chapter in the \emph{Handbook of enumerative combinatorics}~\cite{vatter:survey}.

Formally, let $\sigma$ and $\tau$ be (possibly empty) permutations. A permutation $\pi=\pi(1)\cdots\pi(n)$ of length $n$ is a \emph{juxtaposition of $\sigma$ with $\tau$} if there exists an index $i$ such that $\pi(1)\cdots \pi(i)$ is order-isomorphic to $\sigma$, and $\pi(i+1)\cdots\pi(n)$ is order-isomorphic to $\tau$. The \emph{juxtaposition} of two classes or sets of permutations $\C$ and $\D$ is then defined by
\[ \C\mid\D = \{\pi : \pi \text{ is a juxtaposition of }\sigma\in\C\cup\{\epsilon\}\text{ with }\tau\in\D\cup\{\epsilon\}\}.\]
In this paper, we consider the effect on combinatorial specifications of juxtapositions. The juxtaposition of two permutations can be thought of as a special kind of merge, in which only the \emph{values} of the two permutations can be interleaved arbitrarily, and not the positions. However, our primary motivation in studying juxtapositions is towards a generalisation in another direction, namely the broader study of \emph{grid classes}. 

Our main result is as follows. Definitions are given in Section~\ref{sec-prelim}, but we note here that context-free specifications give rise to algebraic generating functions, while regular ones give rational generating functions.

\begin{theorem}\label{thm:main}
Let $\C$ be a permutation class that admits a bottom-to-top combinatorial specification $\rS$, and let $\M \in\{\Av(21),\Av(12)\}$.
\begin{enumerate}[(i)]
\item\label{item:CM} If $\rS$ tracks the rightmost entry, then there exists a combinatorial specification for $\C\mid\M$ that tracks the rightmost entry. 
\item\label{item:MC} Similarly, if $\rS$ tracks the leftmost entry, then there exists one for $\M\mid\C$ that tracks the leftmost entry. 
\item\label{item:leftright} If $\rS$ tracks both the leftmost and the rightmost entries, then there exists specifications for $\C\mid\M$ and $\M\mid\C$ that do.
\item\label{item:contextfree} If $\rS$ is context-free (resp. regular), then the specifications for $\C\mid\M$ and $\M\mid\C$ are context-free (regular).
\end{enumerate}
\end{theorem}

Since the resulting specifications satisfy the same conditions as the theorem requires of $\C$, the process can be iterated, thereby allowing us to generate combinatorial specifications for classes of the form given in Figure~\ref{fig:k-by-one}.
\begin{figure}[h]
\begin{center}
\begin{tikzpicture}
\draw (0,0) grid (1.2,1);
\draw (2.8,0) grid (6.2,1);
\draw (7.8,0) grid (9,1);
\foreach \i in {1,6} {%
  \draw[dashed] (\i,0) -- (\i+2,0);
  \draw[dashed] (\i,1) -- (\i+2,1);
}
\node[empty] at (0.5,0.5) {$\M_1$};
\node[empty] at (3.5,0.5) {$\M_{j-1}$};
\node[empty] at (4.5,0.5) {$\C$};
\node[empty] at (5.5,0.5) {$\M_{j+1}$};
\node[empty] at (8.5,0.5) {$\M_k$};
\node[empty] at (2,0.5) {$\cdots$};
\node[empty] at (7,0.5) {$\cdots$};
\end{tikzpicture}
\end{center}
\caption{The $k\times 1$ grid classes in Corollary~\ref{cor:skinny}, in which $\C$ possesses a rightmost- and leftmost-entry tracking specification, and $\M_i \in\{\Av(21),\Av(12)\}$ for all $i\neq j$.}
\label{fig:k-by-one}
\end{figure}
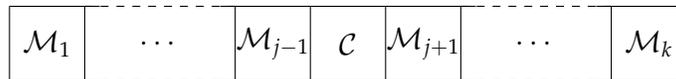

\begin{corollary}\label{cor:skinny}
If a permutation class $\C$ possesses a specification $\rS$ that tracks the rightmost and leftmost entries, then so does any $k\times 1$ grid class of the form given in Figure~\ref{fig:k-by-one}. In particular, if $\C$ possesses an algebraic or rational generating function, then so too does the $k\times 1$ grid class.
\end{corollary}

Permutation classes to which our method applies include any class $\C$ that contains only finitely many simple permutations, but it is not limited to this (for example, in Section~\ref{sec-applicable-classes} we present a suitable specification for $\Av(321)$). By its nature, there is a parallel between our bottom-to-top specifications and the insertion encoding of Albert, Linton and Ru\v{s}kuc~\cite{albert:the-insertion-e:}, and we explore this further in Section~\ref{sec-applicable-classes}.

Prior to this paper, the most general result for $k\times 1$ grids concerns the case where the class $\C$ above is itself also monotone. Such grid classes were called `skinny' grid classes by Bevan~\cite[Chapter 3]{bevan:phd}, where he demonstrates an iterative technique for enumerating any $k\times 1$ grid in which all the cells are monotone. Indeed, the enumeration of such classes was also studied under the term `monotone segment sets' by Albert, Atkinson and Ru\v{s}kuc~\cite{albert:regular-closed-:}, where it was shown that all such classes can be encoded in a regular language, and hence possess rational generating functions. Meanwhile, the process of juxtaposing some permutation class $\C$ with a monotone one plays a key role in the classification of possible growth rates, since each such juxtaposition has a growth rate that is precisely one more than that of $\C$, see Vatter~\cite{vatter:every-growth-rate:} and Bevan~\cite{bevan:intervals}.

More recently, in~\cite{bs:juxt-catalan} the current authors enumerated all juxtapositions of the form $\C\mid\M$ where $\C$ is a `Catalan' class (i.e.\ one of the classes avoiding a single length 3 permutation, all of which are enumerated by the Catalan numbers), and $\M$ is monotone. Our approach in this article is different from the technique employed in that earlier paper, and is considerably more general.

Our method is fully constructive, and relies on operators that act on the equations of the combinatorial specification. Thus, it could be used for practical enumeration tasks, although each application typically results in a sixfold-increase in the number of equations of the specification. Perhaps a more significant consequence of our approach is the general theory, offering a guarantee of the nature of combinatorial specifications and their related enumerations, which could be used to inform future development of algorithms such as \textsc{TileScope} and its underpinning \textsc{CombSpecSearcher}, see Bean's PhD thesis~\cite{bean:thesis}.

The rest of this paper is organised as follows. Following some preliminary definitions and results in Section~\ref{sec-prelim}, the proof of Theorem~\ref{thm:main} is given in Section~\ref{sec-operators}. In Sections~\ref{sec-applicable-classes} and~\ref{sec-examples}, we present a number of classes to which this method applies, and demonstrate the process by computing the specifications and enumerations of three specific juxtaposition classes. Some concluding remarks are given in Section~\ref{sec-conclusion}.

%
%
%
%
%
%
%
%
\section{Combinatorial specifications and permutation classes}\label{sec-prelim}

We will begin by covering a few of the most pertinent  basic definitions relating to the study of permutations, but refer the reader to Bevan's introduction~\cite{bevan2015defs} for others. We say that a permutation $\sigma$ is \emph{contained} in another $\pi$, and write $\sigma\leq\pi$ if there is a subsequence of $\pi$ that is \emph{order-isomorphic} to $\sigma$, i.e.\ a subsequence of the entries of $\pi$ possesses the same relative ordering as the entries of $\sigma$. A \emph{permutation class} is a set of permutations that is closed under the permutation containment ordering, that is, if $\C$ is a permutation class, $\pi\in\C$ and $\sigma\leq \pi$, then we must have $\sigma\in \C$. 

If there is no subsequence of $\pi$ order-isomorphic to $\sigma$, then we say that $\pi$ \emph{avoids} $\sigma$. A permutation class can be defined by its unique minimal avoidance set, which we call the \emph{basis}. We write $\C=\Av(B)$ to mean that $\C$ is the permutation class comprising all permutations that avoid every permutation in $B$,
\[\C = \Av(B) = \{\pi:\pi\not\leq \beta \text{ for all }\beta\in B\}.\]
We will also typically drop the set braces in this notation, thus writing, e.g., $\Av(321)$ instead of $\Av(\{321\})$.


\paragraph{Combinatorial classes and specifications}
A \emph{combinatorial class} is a pair $(\C,|\cdot|)$ where $\C$ is a countable set of combinatorial objects, e.g.~permutations, and $|\cdot|:\C \to \mathbb{Z}^{+}$ is a size function such that the number of objects of size $n$ in $\C$ is finite. In this paper, we will primarily be concerned with combinatorial classes which are sets of permutations.\footnote{In an unfortunate clash of notation, note that a `combinatorial class of permutations' need not be a `permutation class' as it need not be downwards-closed.} However, we will periodically need to consider related objects which will be in bijection with some set of permutations. Formally, we say that two combinatorial classes are \emph{combinatorially isomorphic} if their counting sequences are identical. This is equivalent to the existence of a size-preserving bijection between the two classes.

Following Flajolet and Sedgewick~\cite{flajolet:analytic-combin:}, given a combinatorial class $\C$, a \emph{combinatorial specification} for $\C$ is a system of equations
\[\left\{\begin{array}{r@{}l@{}}
      \C = \S_1 &= f_1(\EE,\Z,\S_1,\ldots,\S_r)\\
      &\vdots\\
      \S_r &= f_r(\EE,\Z,\S_1,\ldots,\S_r)
\end{array}\right.\]
where each $f_i$ is a constructor that only involves \emph{admissible} operations acting on the classes $\S_1,\dots,S_r$, the atom $\Z$ and the empty class $\EE$. The only admissible operations we will consider are the combinatorial sum $+$, Cartesian product $\times$, and the sequence construction $\seq(\cdot)$. Crucially, note that these admissible operations are not necessarily commutative, and indeed we will assume throughout that the Cartesian product is noncommutative.

For convenience, throughout this paper we will write equations such as those in the above specification as $\S_i = f_i(\V)$, where $\V = \{\S_1,\dots,\S_r\}$ is the set of classes. Note that we have omitted the empty class $\EE$ and any atoms: the function $f_i$ may depend on these implicitly.

A priori, the constructors $f_i$ in a combinatorial specification could define recurrences that are (for example) tautological, or impossible to solve. By placing additional conditions on the constructors $f_i$, we can restrict our view to certain types of combinatorial specifications which can always be solved. We highlight two in particular here: 
\begin{description}
\item[Regular]
A combinatorial specification for $\C$ is \emph{regular} if the constructors involve only atoms and the three admissible operations $+$, $\times$ and $\seq(\cdot)$.
\item[Context-free]
A combinatorial specification for $\C$ is \emph{context-free} if the constructors use only the admissible operations $+$ and $\times$.
\end{description}

It is possible to see the family of context-free specifications as more general than regular specifications by observing that the equation $\SZ = \EE + \SZ\times \Z$ is combinatorially isomorphic to $\SZ=\seq(\Z)$, i.e.\ the sequence construction. Indeed, in what follows we will periodically use the class $\SZ$, together with its context-free combinatorial specification, to correspond to $\seq(\Z)$.

The attraction of regular and context-free specifications is due to the following two results:

\begin{proposition}[See, e.g.\ Proposition I.2~\cite{flajolet:analytic-combin:}]\label{prop:regular}
A combinatorial class that has a regular specification admits a rational (ordinary) generating function.
\end{proposition} 

\begin{theorem}[Chomsky-Schutzenberger~\cite{chomsky:the-algebraic-t:}]
\label{thm:chomsky}
A combinatorial class $\C$ that has a context-free specification admits an algebraic (ordinary) generating function $C(z)$. In other words, there exists a (non-null) bivariate polynomial $P(z,y) \in \mathbb{C}[z,y]$ such that $P(z,C(z)) = 0$.
\end{theorem}

\begin{figure}[t]
  \centering
  \begin{tikzpicture}
    \draw (-0.5,0.5) rectangle (0.5,1.5) node[pos=0.5]{$\DD$};
    \draw (-3,-0.5) rectangle (-2,0.5) node[pos=0.5]{$\C$};
    \draw (-1.5,-1.5) rectangle (-0.5,-0.5) node[pos=0.5]{$\C$};
    \filldraw[black] (-2,-2) circle (2pt);
    \draw (-1.6,-2.2) node {$\Z$};
  \end{tikzpicture}
  \caption{An example of a class which would correspond to the term $\Z\C\C\D$ in a bottom-to-top combinatorial specification.}
  \label{fig:order}
\end{figure}
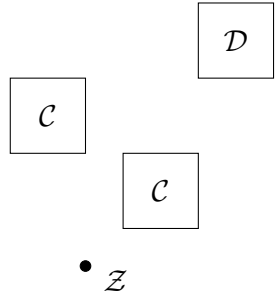

\paragraph{Bottom-to-top specifications}
We say that a combinatorial specification of a permutation class $\C$ is \emph{bottom-to-top} if in the specification, the left-to-right reading order of the atoms and classes in a term corresponds to the order of the entries of the permutations in the class, reading from bottom to top. (Recall that the Cartesian product is noncommutative.) 
For example, the term `$\Z\C\C\D$' in a bottom-to-top specification refers to permutations which, from bottom to top, can be seen as an atom $\Z$, followed by two instances of elements from $\C$, followed by an element from $\D$. A typical example of how this might arise is illustrated in Figure~\ref{fig:order}.

\subsection{Greedy griddings and atoms}
We will be working with operators that act on bottom-to-top combinatorial specifications of a class $\C$, and produce new combinatorial specifications for the class $\C\mid\M$ or $\M\mid\C$ (i.e.\ for $\C$ juxtaposed by a monotone class $\M$ to its right or left). A challenge in this process is to ensure that each permutation in such a juxtaposition has a unique representation in the resulting combinatorial specification. Our solution will be to greedily add as many entries as possible into the monotone class $\M$. For this to work, the operators need to know whether the bottom-to-top reading of the specification has encountered the rightmost (for $\C\mid\M$) or leftmost (for $\M\mid\C$) entry of $\C$. See Figure~\ref{fig:leftmostgridline} for an illustration.
\begin{figure}
  \centering
  \begin{tikzpicture}[scale=0.2]
\plotperm{9,1,7,2,11,26,16,12,17,23,18,25,24,3,4,5,6,8,10,13,14,15,19,20,21,22}
\node[permpt,fill=black!0] at (24) {};
\draw[black, thick] (13.5,0.5) -- (13.5,26.5);
\end{tikzpicture}
\caption{The role of the rightmost point of the class $\C$ in fixing a greedy gridding for $\C\mid\M$. A gridline any further left would produce a basis element of $\M$ involving this entry.}
\label{fig:leftmostgridline}
\end{figure}
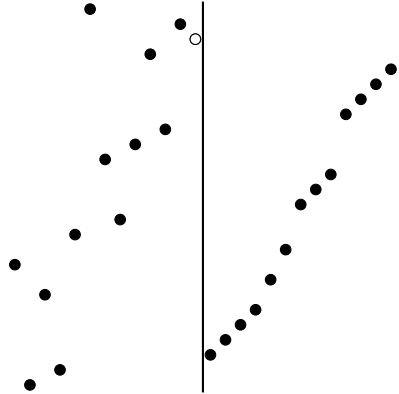

To this end, we will use four different atoms: In addition to the standard atom $\Z$ to represent generic entries of a permutation, we may use $\ZL$ to identify the left-most entry and $\ZR$ the right-most entry. The fourth atom, $\ZLR$, represents an entry that is simultaneously both leftmost and rightmost, and thus can only occur to represent the singleton permutation.

Clearly, each term in a combinatorial specification can have at most one leftmost atom and one rightmost atom, and it will be helpful to decorate the classes in the specification with $\mathsf{L}$, $\mathsf{R}$ or $\mathsf{LR}$ to indicate when these atoms are present or not present. For example, $\CL$ represents a class in which the leftmost atom $\ZL$ is used exactly once in each permutation (but $\ZR$ is not used).

We say that a class $\C$ admits a combinatorial specification that \emph{tracks the rightmost entry} if there exists a bottom-to-top specification for $\CR$. Note that in moving from combinatorial specifications to generating functions, the four atoms all make the same contribution $z$. Thus Proposition~\ref{prop:regular} and Theorem~\ref{thm:chomsky} remain valid, and we can still use the terms `regular' and `context-free' to refer to specifications involving all the different atoms.

\section{Operators and the proof of Theorem~\ref{thm:main}}\label{sec-operators}

To describe the process by which we transform a combinatorial specification for $\C$ into one for a juxtaposition of $\C$ with a monotone class $\M$, we begin by considering only the juxtaposition $\C\mid\Av(21)$. We will later describe symmetry operators that allow us to obtain the other possible juxtapositions.

We can think of the juxtaposition of $\C$ with $\Av(21)$ as the entries of $\C$, with sequences of monotone entries inserted vertically in-between (but horizontally to the right). Let $x$ be any entry of a permutation $\pi\in\C$. The \emph{gap} associated with $x$ is the vertical space immediately below $x$ and above the entry immediately below (if it exists) -- see Figure~\ref{fig:xregion}. In addition, the \emph{top gap} is the vertical space above the highest point of $\C$.

\begin{figure}[ht]
  \centering
  \begin{tikzpicture}
  \node[permpt] at (1,2) {};
  \node[permpt] at (2,4) {};
  \node[permpt] at (3,1) {};
  \node[permpt,label=above left:$x$] at (4,3) {};
    
    \draw[draw=none, thick, pattern=north west lines] (4.5,2) rectangle (6.5,3);
    \draw[black, thick] (4.5,0) -- (4.5,4.5); 
    \draw[thin] (2.2,4)--(6.5,4);
    \draw[thin] (4.2,3)--(6.5,3);
    \draw[thin] (1.2,2)--(6.5,2);
    \draw[thin] (3.2,1)--(6.5,1);
  \end{tikzpicture}
  \caption{The shaded region on the RHS corresponds to the gap associated with $x$.}
  \label{fig:xregion}
\end{figure}
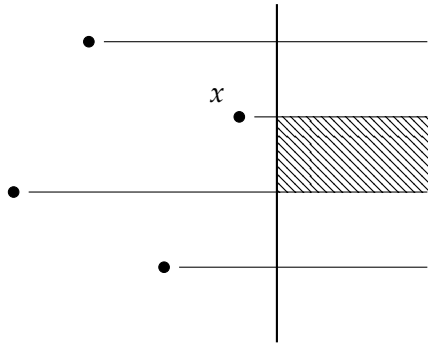

\begin{lemma}\label{lem:zr}
Let $\C$ be a permutation class with a combinatorial specification that tracks the rightmost entry. Then the juxtaposition $\C\mid\Av(21)$ is combinatorially isomorphic to the class $\C$ with (possibly empty) sequences of atoms inserted in the gaps, such that there is at least one nonempty gap no higher than the atom $\ZR$ of $\C$.
\end{lemma}

\begin{proof}
First, by the comments made before the statement of the lemma it is clear that the juxtaposition $\C\mid\Av(21)$ can be thought of as sequences of monotone entries inserted into the gaps (including the top gap) of a (possibly empty) permutation from $\C$. 

It remains to ensure that we consider each permutation in $\C\mid\Av(21)$ at most once, and for this we use the greedy gridding described in the previous section to include as many entries in the monotone part as possible. Let $\pi\in\C\mid\Av(21)$, and write $\pi = \sigma\tau$ such that  $\sigma\in\C$, and $\tau$ is as large as possible subject to $\tau\in\Av(21)$.  First, if $\sigma$ is empty, then $\pi=\tau\in\Av(21)$ is either empty or an increasing permutation, which can be thought of as the insertion of a sequence into the top gap of the empty permutation $\epsilon$.

Otherwise, both $\sigma$ and $\tau$ are nonempty, and thus the rightmost entry of $\sigma$ must be above at least one entry of $\tau$. Since $\ZR$ tracks the rightmost entry of $\sigma$ in a bottom-to-top specification, the greedy gridding condition translates to the requirement that there is at least one entry of $\tau$ in some gap below $\ZR$, as required.
\end{proof}

To produce the combinatorial specification representing this insertion of sequences of entries into the gaps of permutations from $\C$, we will define operators that are sensitive to the atom $\ZR$ to ensure the conditions of Lemma~\ref{lem:zr} are satisfied. Specifically, we will distinguish between an operator which inserts the lowest entry in $\Av(21)$, and operators that insert (possibly empty) sequences of entries above this. 

Furthermore, the claim made in Theorem~\ref{thm:main} requires us to output a combinatorial specification for $\C\mid\Av(21)$ that tracks the rightmost point. By our choice of greedy gridding this rightmost point must lie in $\Av(21)$, and is therefore created by the process of applying operators. Thus, we will also distinguish between an operator which inserts the rightmost (${}={}$ highest) entry $\ZR$ in some gap, and an operator that acts on gaps below this.

To this end, we define the following linear operators. 
\begin{description}
\item[$\opo$] the null operator: inserts no entries, but replaces $\ZR$ with $\Z$ in its operand (if it appears).
\item[$\opi$] inserts a single point (lowest and highest entry) in some gap of its operand below $\ZR$.
\item[$\opoo$] inserts (possibly empty) increasing sequences of entries in the gaps of its operand.
\item[$\opio$] inserts the lowest entry in some gap of its operand (below the operand atom $\ZR$), followed by (possibly empty) increasing sequences of entries in the subsequent gaps.
\item[$\opoi$] inserts (possibly empty) increasing sequences of entries in the gaps of its operand, finishing with the insertion of a new rightmost atom $\ZR$ in some (non-topmost) gap.
\item[$\opii$] inserts the lowest entry in some gap of its operand (below the operand atom $\ZR$), followed by (possibly empty) increasing sequences of entries in the subsequent gaps, and finishing with the insertion of a new rightmost atom $\ZR$.
\end{description}
For convenience later, let $\bop = \{\opo,\opi,\opoo,\opio,\opoi,\opii\}$ denote the set of operators.

We will defer the description of the action of the operators in $\bop$ until later, except to define their action on the neutral atom $\EE$. The two operators that add potentially empty sequences of points, namely $\opo$ and $\opoo$, act as the identity on $\EE$, whereas the other four annihilate the term entirely; this convention is simply to ensure that later expressions involving combinatorial sums and Cartesian products are consistent.
\begin{align*}
\opo(\EE) &= \opoo(\EE) = \EE\\
\opi(\EE) &= \opio(\EE) = \opoi(\EE) = \opii(\EE) = 0.
\end{align*}

\begin{lemma}\label{lem:CM}
Let $\C$ be a class given by a combinatorial specification which tracks the rightmost entry, and let $\M=\Av(21)$. Then the class $\C\mid\M$ is given by \[\C\mid\M=\EE+\SZ\times\ZR+\opi(\C)+\opii(\C)+\opio(\C)\times \SZ\times\ZR,\]
where $\SZ = \EE + \SZ\times\Z = \seq(\Z)$.

Furthermore, if $\C=\CL$ also tracks the leftmost entry, then the class $\CL\mid\M$ is given by
\[\CL\mid\M=\EE+\ZLR +\ZL\times\SZ\times\ZR+\opi(\CL)+\opii(\CL)+\opio(\CL)\times \SZ\times\ZR.\]
\end{lemma}

\begin{proof}
Let $\pi\in \C\mid\M$ be nonempty, and take the leftmost gridding $\pi=\pi_\C\pi_\M$. If $\pi_\C$ is empty, then $\pi$ is a monotone increasing sequence of points, and this is captured by the term $\SZ\times\ZR$, so we now assume that $\pi_\C$ is non-empty. Observe that, by the greedy gridding, there is at least one entry in $\pi_\M$. If there is precisely one, then there must be at least one entry of $\pi_\C$ above it, and this is captured by the term $\opi(\C)$.

We may now suppose that there are at least two entries in $\pi_\M$. If all the entries in $\M$ lie in gaps associated with entries from $\pi_\C$, then $\pi$ lies in $\opii(\C)$. 

On the other hand, if one or more entries lie in the top gap, then $\ZR$ lies in this gap and is the topmost entry, preceded by a (possibly empty) sequence of atoms. Below the top gap, we must insert the lowest entry of $\pi_\M$ (which must be below the rightmost entry of $\pi_\C$), followed by a (possibly empty) sequence, interleaved with the entries of $\pi_\C$. This accounts for the final term.

The case where $\C=\CL$ also tracks the leftmost entry follows similarly; indeed, the only difference is in handling the case where the portion $\pi_\C$ of $\pi$ is empty.
\end{proof}

Thus in order to describe the class $\C\mid\Av(21)$, we need to understand terms such as $\opi(\C)$, $\opii(\C)$ and $\opio(\C)$. To do this, we can use the given combinatorial specification for the class $\C$. In other words, we will be applying our operators to equations, each of which is built from a combination of $+$, $\times$ and $\seq(\cdot)$ of other combinatorial classes (which have their own specifications) and atoms.  

We need to document precisely the effect these six operators have on their operands; this is carried out in the next three subsections. As the operators act linearly, and having described the effect on the neutral atom $\EE$ earlier, we have three main tasks: to describe (1) the action on the four atoms, (2) the action over Cartesian products, and (3) the action over the sequence constructor.

\subsection{The action on the four atoms}\label{subsec:atoms}
The complete list of actions is given below, followed by a short narrative to explain the terms. Recall that $\SZ = \EE+\SZ\Z = \seq(\Z)$ represents sequences of the atom $\Z$. 
\[
\begin{array}{c|c|c|c|c|c|c}
\text{Operand}&\opo&\opi&\opoo&\opio&\opoi&\opii\\\hline
\Z&\Z&\ZR\Z&\SZ\Z&\Z\SZ\Z&\SZ\ZR\Z&\Z\SZ\ZR\Z\\
\ZR&\Z&\ZR\Z&\SZ\Z&\Z\SZ\Z&\SZ\ZR\Z&\Z\SZ\ZR\Z\\\hline
\ZL&\ZL&\ZR\ZL&\SZ\ZL&\Z\SZ\ZL&\SZ\ZR\ZL&\Z\SZ\ZR\ZL\\
\ZLR&\ZL&\ZR\ZL&\SZ\ZL&\Z\SZ\ZL&\SZ\ZR\ZL&\Z\SZ\ZR\ZL\\
\end{array}
\]

For an operand that is one of the four atoms, the operators $\opoo$, $\opio$, $\opoi$ and $\opii$ all add a (possibly empty) sequence of entries in the gap corresponding to the atom -- this is represented by $\SZ$. Operators that insert a new rightmost entry (namely $\opi$, $\opoi$ and $\opii$) insert the atom $\ZR$ immediately before the final atom $\Z$ (which corresponds to the same entry as the original operand). 

Finally, we have also listed explicitly the action on the operands $\ZL$ and $\ZLR$, although these are broadly analogous to $\Z$ and $\ZR$. Such atoms will still be leftmost after the effect of the operator, but they will cease to be rightmost (even in the case of $\opo$).

\subsection{The action on Cartesian products}\label{subsec:cartesian}

Whereas the operators all act linearly over combinatorial sums ($+$), their action on the Cartesian product of two (or more) classes is dependent upon whether they insert the lowest entry of the new cell or not. We will call those that do not insert the lowest entry ($\opo$, $\opoo$ and $\opoi$),  \emph{$\ZR$-invariant}  while the others are $\ZR$-sensitive. We will describe the action of these two groups separately.

Note also that we have not mentioned explicitly the leftmost atom $\ZL$ (or the class containing it) of the operand. This is because the action on the Cartesian product of two or more classes is the same whether $\ZL$ is present or not. 

\paragraph{$\ZR$-invariant operators}
Operators which do not insert the lowest entry in the new cell act in the same way irrespective of whether the operand contains $\ZR$ or not. There are three such operators, namely $\opo$, $\opoo$ and $\opoi$. Their action over Cartesian products is therefore defined by:
\begin{align*}
\opo(\A\B) &= \opo(\A)\opo(\B)\\
\opoo(\A\B) &= \opoo(\A)\opoo(\B)\\
\opoi(\A\B) &= \opoi(\A)\opo(\B)+\opoo(\A)\opoi(\B)
\end{align*}
for all classes $\A$ and $\B$.\footnote{Note that by the convention for the actions of these operators of the empty class $\EE$, these and later definitions for the Cartesian product are consistent with the trivial products $\A\times \EE$ and $\EE\times\B$.}

\paragraph{$\ZR$-sensitive operators}
The remaining three operators all insert the lowest entry of the cell, and are therefore sensitive to the position of the atom $\ZR$. For the action on the product of two operators, we distinguish two cases: $\A\times\B$ where no term of $\A$ contains $\ZR$, and $\AR\times\B$ where every term of $\AR$ contains exactly one $\ZR$ (or another class $\D^{\mathsf{R}}$ of the same form).
\begin{align*}
\opi(\A\B)	&=\opi(\A)\opo(\B)+\opo(\A)\opi(\B)\\
\opio(\A\B)	&=\opio(\A)\opoo(\B)+\opo(\A)\opio(\B)\\
\opii(\A\B)	&=\opii(\A)\opo(\B)+\opio(\A)\opoi(\B)+\opo(\A)\opii(\B)\\[5pt]
\opi(\AR\B)	&=\opi(\AR)\opo(\B)\\
\opio(\AR\B)&=\opio(\AR)\opoo(\B)\\
\opii(\AR\B)&=\opii(\AR)\opo(\B)+\opio(\AR)\opoi(\B)
\end{align*}
Note, in each case, that the only difference between the action on $\A\B$ and the action on $\AR\B$ is to lose the final term (involving $\opo(\A)$). Such a term would correspond to inserting the lowest entry above the atom $\ZR$, which is of course disallowed by the greedy gridding.

%
%
%
%
\subsection{The action on sequences}
In the case that we are given a regular specification, then we need to understand how the operators in $\bop$ act on the sequence constructor, $\seq$. Note first that the atoms $\ZL$, $\ZR$ and $\ZLR$ can \emph{never} appear inside a $\seq$ operator, since there can only ever be at most one copy of each of these types of atom in a term. Thus, in describing these actions we do not need to consider the $\ZR$-sensitive operators separately.\footnote{Note that it would be perfectly possible to describe the action of operators in $\bop$ on sequence constructors that involve, e.g. $\ZR$.}

Let $\A$ be an atom, class, or an expression involving combinatorial sums, Cartesian products and sequences of the atom $\Z$ and other classes not involving $\ZL$, $\ZR$ or $\ZLR$. It is routine to verify that the following expressions correctly describe the action of operators on the sequence construction.
\begin{align*}
\opo(\seq(\A)) &= \seq(\opo(\A))\\
\opoo(\seq(\A)) &= \seq(\opoo(\A))\\
\opi(\seq(\A)) &= \seq(\opo(\A))\times\opi(\A)\times\seq(\opo(\A))\\
\opio(\seq(\A)) &= \seq(\opo(\A)) \times \opio(\A) \times \seq(\opoo(\A))\\
\opoi(\seq(\A)) &= \seq(\opoo(\A)) \times \opoi(\A) \times \seq(\opo(\A))\\
\opii(\seq(\A)) &= \seq(\opo(\A)) \times \opio(\A) \times \seq(\opoo(\A)) \times \opoi(\A) \times \seq(\opo(\A))\\
&\qquad {}+ \seq(\opo(\A)) \times \opii(\A) \times \seq(\opo(\A))
\end{align*}

%
%
%
%
\subsection{Expansions}
We have now fully described how the operators act on individual terms, across Cartesian products and on the sequence operator. In this subsection we will apply these descriptions to the equations from the combinatorial specification for a class $\C$, in order to produce new equations for the juxtaposition class $\C\mid\Av(21)$. 

Let $\V$ be a collection of combinatorial classes, and let $\Omega\in\bop$ be an operator. Define $\V_{\Omega} = \{ \Omega(\S) : \S\in\V\}$ to be the collection of combinatorial classes under the action of $\Omega$, and $\V_{\bop} = \bigcup_{\Omega\in\bop} V_{\Omega}$. Our aim is to produce a combinatorial specification for $\C\mid \Av(21)$, using the symbols in $\V_{\bop}$ together with the equations in Lemma~\ref{lem:CM}.

Given an equation $\C = f(\V)$ from some combinatorial specification, we have (trivially), $\Omega(\C) = \Omega(f(\V))$. The \emph{expansion} of $\Omega(\C)$ is then the equation obtained by using the linearity of the operators and the expressions for atoms and Cartesian products given above, i.e.\ the equation of the form
\[\Omega(\C) = g(\V_{\bop}\cup\{\SZ\}),\]
for some function $g$. See Figure~\ref{fig:opiiexpansion} for an example. The properties of $g$ that we need are given in the next lemma.

\begin{figure}[!ht]
\newcommand{\drawclasspicture}[2]{
    \draw[dashed] (2,4) -- (5,4);
    \draw[dashed] (2,3) -- (5,3);
    \draw[dashed] (1,2) -- (5,2);
    \draw[dashed] (1,1) -- (5,1);
    \draw[dashed] (3,0) -- (5,0);
    \draw[fill=gray!30] (1,3) rectangle ++(1,1) node[pos=0.5]{$\D$};
    \draw[fill=gray!30] (3,2) rectangle ++(1,1) node[pos=0.5]{$\CR$};
    \draw[fill=gray!30] (0,1) rectangle ++(1,1) node[pos=0.5]{$\B$};
    \draw[fill=gray!30] (2,0) rectangle ++(1,1) node[pos=0.5]{$\A$};
    \draw[-, very thick] (4,-0.5) -- (4,4.5);
    \node[permpt] at (4,2.3) {};
    \node[permpt] (L) at (4.3,#1) {};
    \node[permpt] (T) at (5.0,#2) {};
    \draw[dotted] (L) -- (T);%
    \path[use as bounding box] (-1.05,-1.5) -- (6.05,4.5);
}%
  \centering
\begin{tikzpicture}[scale=0.6]
    \drawclasspicture{0.1}{0.8}
    \draw (2.5,-1.5) node[above]{\footnotesize $\opii(\A)\opo(\B\CR\D)$};
  \end{tikzpicture}%
\begin{tikzpicture}[scale=0.6]
    \drawclasspicture{0.1}{1.8}
    \draw (2.5,-1.5) node[above]{\footnotesize $\opio(\A)\opoi(\B)\opo(\CR\D)$};
\end{tikzpicture}%
\begin{tikzpicture}[scale=0.6]
  \drawclasspicture{0.1}{2.8}
    \draw (2.5,-1.5) node[above]{\footnotesize $\opio(\A)\opoo(\B)\opoi(\CR)\opo(\D)$};
\end{tikzpicture}%
\begin{tikzpicture}[scale=0.6]
    \drawclasspicture{0.1}{3.8}
    \draw (2.5,-1.5) node[above]{\footnotesize $\opio(\A)\opoo(\B\CR)\opoi(\D)$};
\end{tikzpicture}%
\\[10pt]
  \begin{tikzpicture}[scale=0.6]
    \drawclasspicture{1.1}{1.8}
    \draw (2.5,-1.5) node[above]{\footnotesize $\opo(\A)\opii(\B)\opo(\CR\D)$};
  \end{tikzpicture}
  \begin{tikzpicture}[scale=0.6]
    \drawclasspicture{1.1}{2.8}
    \draw (2.5,-1.5) node[above]{\footnotesize $\opo(\A)\opio(\B)\opoi(\CR)\opo(\D)$};
  \end{tikzpicture}
   \begin{tikzpicture}[scale=0.6]
    \drawclasspicture{1.1}{3.8}
    \draw (3,-1.5) node[above]{\footnotesize $\opo(\A)\opio(\B)\opoo(\CR)\opoi(\D)$};
  \end{tikzpicture}
\\[10pt]
  \begin{tikzpicture}[scale=0.6]
    \drawclasspicture{2.1}{2.8}
    \draw (2.5,-1.5) node[above]{\footnotesize $\opo(\A\B)\opii(\CR)\opo(\D)$};
  \end{tikzpicture}
   \begin{tikzpicture}[scale=0.6]
    \drawclasspicture{2.1}{3.8}
    \draw (2.5,-1.5) node[above]{\footnotesize $\opo(\A\B)\opio(\CR)\opoi(\D)$};
  \end{tikzpicture}
\caption{The expansion of $\opii(\C)$ given the equation $\C =\A\B\CR\D$ is given by the combinatorial sum of the 9 terms above.}
\label{fig:opiiexpansion}
\end{figure}
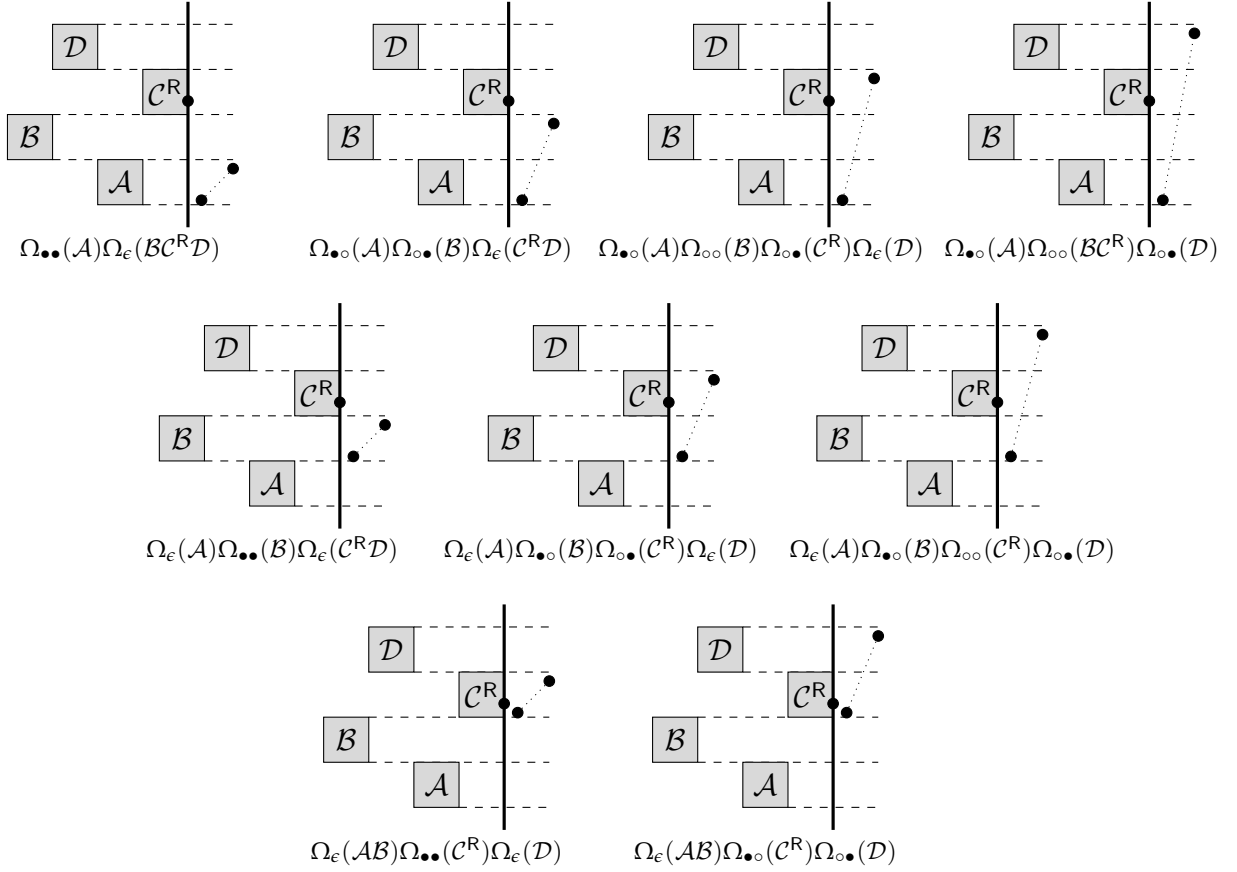

\begin{lemma}\label{lem:expansion}
Let $\C = f(\V)$ be an equation from a combinatorial specification using only the combinatorial sum, Cartesian product and sequence constructors. Then, for every operator $\Omega\in\bop$, the expansion of $\Omega(\C)$ is an equation involving only the combinatorial sum, Cartesian product and sequence constructors, acting only on atoms and classes from $\V_{\bop}\cup\{\SZ\}$.

Furthermore, if $\C=f(\V)$ does not use the sequence constructor, then neither does the expansion of $\Omega(\C)$.
\end{lemma}

\begin{proof}
First, any operator $\Omega\in\bop$ acts linearly on combinatorial sums
\[\Omega(\A+\B) = \Omega(\A) + \Omega(\B),\]
so it suffices to consider the action of $\Omega$ on individual terms of $f(\V)$. Such terms, however, are themselves composed of Cartesian products of classes in $\V$ or sequences of classes in $\V$, and the rules for any operator in $\bop$ described earlier in this subsection demonstrate that these can be rewritten as expressions involving combinatorial sums, Cartesian products and sequences of classes in $\V_{\bop}$. Note that the action of an operator $\Omega\in\bop$ on an atom is described in Subsection~\ref{subsec:atoms}, and can be expressed as the Cartesian product of atoms and the class $\SZ$. This completes the proof.

Finally, if there are no sequence constructors in the original equation, then the rules given earlier for operators in $\bop$ guarantee that no sequence constructor will be introduced in the expansion of $\Omega(\C)$. 
\end{proof}

The result above tells us that we can create equations for the classes in $\V_{\bop}$ from the equations in the combinatorial specification for $\C$, and that these equations only use our restricted set of admissible operations. From these equations, we can now describe a combinatorial specification for $\C\mid\Av(21)$, which is the first key piece of our main result.

\begin{proposition}\label{prop:CMcombspec}
Let $\C$ be a permutation class, and $\rS$ a combinatorial specification for $\C$ that tracks the rightmost entry. Then there exists a combinatorial specification $\rS'$ for the class $\C\mid\Av(21)$ that tracks the rightmost entry.

Furthermore, if $\rS$ is context-free, and/or tracks the leftmost entry, then so is $\rS'$. 
\end{proposition}

\begin{proof}
First, let $\V$ denote the set of classes that appear on the left hand side of any equation in the specification $\rS$. Thus, any equation in $\rS$ is of the form $\S = f(\V)$. 

By Lemma~\ref{lem:CM}, the class $\C\mid\Av(21)$ is given by the equation
\[\C\mid\Av(21)=\EE+\SZ\times\ZR+\opi(\C)+\opii(\C)+\opio(\C)\times \SZ\times\ZR.\]
We therefore need to establish a system of equations to describe the classes $\opi(\C)$, $\opii(\C)$ and $\opio(\C)$, which taken together with the equation $\SZ = \EE + \SZ\Z$, forms a complete combinatorial specification. We claim that it suffices to define the system of equations $\rS'$ to comprise the following:
\begin{align*}
\C\mid\Av(21) &=\EE+\SZ\times\ZR+\opi(\C)+\opii(\C)+\opio(\C)\times \SZ\times\ZR\\
\Omega(\S) &= g_{\Omega(\S)}(\V_{\bop}\cup\{\SZ\}) &\text{for every }\Omega\in\bop, \S\in\V\\
\SZ &= \EE+\SZ\times \Z.
\end{align*}

Every class in $\V_{\bop}$ is of the form $\Omega(\S)$ for some $\Omega\in\bop$ and $\S\in\V$. Writing $\S=f(\V)$ (an equation taken from $\rS$), by Lemma~\ref{lem:expansion} the expansion of $\Omega(\S)$ is an equation involving only combinatorial sums and Cartesian products of classes from $\V_{\bop}\cup\{\SZ\}$. Thus the system $\rS'$ above is indeed a context-free combinatorial specification, providing $\rS$ is. Furthermore, each term in the equation for $\C\mid\M$ acts from bottom-to-top, and specifies the position of the rightmost entry $\ZR$.

If the atom $\ZL$ and/or $\ZLR$ is present in the specification for $\C(=\CL)$, then, by definition, the operators in $\bop$ ensure that $\ZL$ and/or $\ZLR$ are correctly tracked for $\CL\mid\M$. However, in this case we need to use the following expression for the class $\CL\mid\M$, given in Lemma~\ref{lem:CM}:
\[\CL\mid\Av(21)=\EE+\ZLR +\ZL\times\SZ\times\ZR+\opi(\CL)+\opii(\CL)+\opio(\CL)\times \SZ\times\ZR.\]
\end{proof}

Analogously, we have the following result that guarantees a regular specification is obtained if the initial specification is regular.

\begin{proposition}\label{prop:CMratcombspec}
Let $\C$ be a permutation class, and $\rS$ a regular combinatorial specification for $\C$ that tracks the rightmost entry. Then there exists a regular combinatorial specification $\rS'$ that tracks the rightmost entry for $\C\mid\Av(21)$.

Furthermore, if $\rS$ tracks the leftmost entry, then so does $\rS'$.
\end{proposition}

\begin{proof}
We construct the specification $\rS'$ from $\rS$ following the proof of Proposition~\ref{prop:CMcombspec}. We will omit the details as they are similar, the only exception being that we now need to use the defined actions of operators in $\bop$ on sequences in the expansions. Note that $\rS'$ will track the rightmost entry, and if $\rS$ tracks the leftmost entry, then so does $\rS'$.

It remains to verify that $\rS'$ is regular. Note that the right hand side of any equation in $\rS$ involves only atoms, combined using the sum, Cartesian product and sequence operations. As such, by considering the effect of the operators in $\bop$, the expansion of an equation in $\rS$ under some operator in $\bop$ will yield an equation in $\rS'$ also involving only sums, products and sequences of atoms, except for the appearance of the class $\SZ$. This however, can be replaced by $\seq(\Z)$ wherever it appears, whence the equation in $\rS'$ satisfies the conditions to be regular.
\end{proof}

%
%
%
%
%
\subsection{Symmetry operators}

Having established how to find a combinatorial specification for $\C\mid\Av(21)$ given one for $\C$, we now need to describe how to find specifications for $\C\mid\Av(12)$, $\Av(21)\mid\C$ and $\Av(12)\mid\C$. 

Let $\pi$ be a permutation of length $n$. The \emph{reverse} of $\pi$, denoted $\pi^r$, is the permutation obtained by reading the entries of $\pi$ in reverse order, i.e. $\pi^r(i) =\pi(n+1-i)$. Similarly, the \emph{complement} of $\pi$ is $\pi^c$, defined by $\pi^c(i) = n+1-\pi(i)$. Combining these two operations, we obtain the \emph{reverse-complement}, $\pi^{rc}$, which is formally defined by $\pi^{rc}(i) = n+1 -\pi(n+1-i)$.

By extension, we may consider the reverse, complement, and reverse-complement of a class $\C$ of permutations, for example $\C^r = \{\pi^r : \pi\in\C\}$. 

\begin{observation}\label{obs:symmetry}
For any permutation class $\C$, we have the following:
\begin{align*}
\C\mid\Av(12) &= (\C^c\mid\Av(21))^c\\
\Av(12)\mid\C &= (\C^r\mid\Av(21))^r, \text{ and}\\
\Av(21)\mid\C &= (\C^{rc}\mid\Av(21))^{rc}.\end{align*}
\end{observation}

Given a bottom-to-top combinatorial specification for $\C$, we can obtain a bottom-to-top specification for $\C^c$ simply by reversing the order in which terms in a Cartesian product are taken. This is characterised by the \emph{complement} operator, $\Theta$, which acts linearly on combinatorial sums, as the identity on atoms (i.e.\ $\Theta(\Z)=\Z$, $\Theta(\ZR)=\ZR$, $\Theta(\ZL)=\ZL$, $\Theta(\ZLR)=\ZLR$), and is then defined recursively on Cartesian products as follows:
\[\Theta( \A\times \B) =  \Theta(\B)\times\Theta(\A)\]
for classes $\A$ and $\B$. The action on the sequence operator is given by
\[\Theta(\seq(\A)) = \seq(\Theta(\A)).\]

Similarly, given a specification for $\C$ that tracks both the leftmost and rightmost entries, we can find one for $\C^r$ by preserving the order in the Cartesian product, but switching the atoms $\ZR$ and $\ZL$. This is characterised by the \emph{reverse} operator, $\Phi$, which acts linearly on combinatorial sums, distributively on Cartesian products ($\Phi(\A\times\B)=\Phi(\A)\times\Phi(\B)$) and sequences ($\Phi(\seq(\A)) = \seq(\Phi(\A))$), and satisfies
\[\Phi(\Z)=\Z,\quad \Phi(\ZR)=\ZL,\quad \Phi(\ZL)=\ZR,\quad \Phi(\ZLR)=\ZLR.\]
Note that if a class $\C$ tracks only the leftmost entry, then $\Phi(\C)$ tracks only the rightmost entry and vice versa, and if $\C$ tracks both, then so does $\Phi(\C)$.

Given some equation $\C = f(\V)$ from a combinatorial specification, the \emph{expansion} of $\Theta(\C)$ is the equation obtained by using the linearity of $\Theta$, together with the defined action on Cartesian products, to obtain an equation of the form
\[\Theta(\C) = g(\V_{\Theta})\]
for some $g$. Similarly, the expansion of $\Phi(\C)$ is the equation obtained following the definition of $\Phi$ above, $\Phi(\C)=h(\V_{\Phi})$ for some $h$.

We now have the following observation, mirroring Lemma~\ref{lem:expansion}. We omit its proof.

\begin{lemma}
Let $\C = f(\V)$ be an equation from a combinatorial specification using only the combinatorial sum, Cartesian product, and sequence constructors. Then the expansion of $\Theta(\C)$ and $\Phi(\C)$ are equations involving only the combinatorial sum, Cartesian product and sequence operators, acting only on classes from $\V_\Theta$ and $\V_\Phi$, respectively.

Furthermore, if $\C=f(\V)$ does not contain any instance of a sequence operator, then neither does the expansion of $\Theta(\C)$.
\end{lemma}

Similarly, we may combine these expansions to obtain combinatorial specifications for $\C^c$, $\C^r$ and $\C^{rc}$. As its proof resembles (but is simpler than) the proof of Proposition~\ref{prop:CMcombspec}, we will omit it.

\begin{proposition}\label{prop:TPcombspec}
Let $\rS$ be a combinatorial specification for a class $\C$. Then there exist combinatorial specifications $\rS^\Theta$, $\rS^\Phi$ and $\rS^{\Theta\cdot\Phi}$  for the classes $\C^c$, $\C^r$ and $\C^{rc}$, respectively.

Furthermore, if $\rS$ tracks the rightmost (resp. leftmost) entry, then 
so does $\rS^\Theta$, while the specifications $\rS^{\Theta\cdot\Phi}$ and $\rS^\Phi$ track the leftmost (resp. rightmost) entry.
\end{proposition}

Equipped with Observation~\ref{obs:symmetry} and combinatorial specifications for $\C\mid\M$, $\C^c$, $\C^r$ and $\C^{rc}$, we can now complete the proof of Theorem~\ref{thm:main}.

\begin{proof}[Proof of Theorem~\ref{thm:main}]
By Proposition~\ref{prop:CMcombspec}, if the combinatorial specification $\rS$ for $\C$ tracks the rightmost entry, then there exists a specification for $\C\mid\Av(21)$ that does. To complete the proof of part~(\ref{item:CM}), therefore, we need to consider $\C\mid\Av(12)$.

By Observation~\ref{obs:symmetry}, we have $\C\mid\Av(12) = (\C^c\mid\Av(21))^c$. By Proposition~\ref{prop:TPcombspec}, there exists a combinatorial specification $\rS^\Theta$ for $\C^c$. To this we apply Proposition~\ref{prop:CMcombspec} to obtain a combinatorial specification for $\C^c\mid\Av(21)$. Finally, an application of $\Theta$ to this specification gives us a specification for $(\C^c\mid\Av(21))^c$, which is combinatorially isomorphic to $\C\mid\Av(12)$. Furthermore, all these specifications track the rightmost entry by Propositions~\ref{prop:CMcombspec} and~\ref{prop:TPcombspec}.

Similar arguments can now be made to establish part~(\ref{item:MC}) of the theorem: Observation~\ref{obs:symmetry} gives us equations for $\Av(21)\mid\C$ and $\Av(12)\mid\C$, and combinatorial specifications for these can then be computed using applications of $\Phi$, $\Theta$, and the operators in $\bop$. Note that if $\rS$ tracks the leftmost entry, then by Proposition~\ref{prop:TPcombspec} $\rS^\Phi$ and $\rS^{\Theta\cdot\Phi}$ (being specifications $\C^r$ and $\C^{rc}$) track the rightmost entry (as required for the operators in $\bop$ to append $\Av(21)$ to the right). After appending $\Av(21)$, a further application of $\Phi$ replaces the rightmost entry with the leftmost, while any application of $\Theta$ keeps the rightmost and leftmost atoms intact.

For part~(\ref{item:leftright}), note that the operators $\Theta$ and those in $\bop$ preserve the atom $\ZL$, while $\Phi$ interchanges the atoms $\ZL$ and $\ZR$, and both $\Theta$ and $\Phi$ preserve $\ZLR$. Combined with Proposition~\ref{prop:CMcombspec},  We conclude that the resulting combinatorial specifications for all of $\C\mid\Av(21)$, $\C\mid\Av(12)$, $\Av(21)\mid\C$ and $\Av(12)\mid\C$ also track both the leftmost and rightmost entries.

Finally, for part~(\ref{item:contextfree}), if $\rS$ is context-free or rational, then it is clear that the operators $\Theta$ and $\Phi$ preserve this. We know that the operators in $\bop$ preserve context-freeness by Proposition~\ref{prop:CMcombspec}, and preserve rationality by Proposition~\ref{prop:CMratcombspec}. This completes the proof.
\end{proof}

%
%
%
%
%
%
\section{Applicable classes}\label{sec-applicable-classes}

In this section, we demonstrate classes which possess combinatorial specifications that track the rightmost and leftmost entries.

\paragraph{Classes containing finitely many simple permutations}

It has been known since Albert and Atkinson~\cite{albert:simple-permutat:} that classes which possess only finitely many simple permutations admit an algebraic generating function. Subsequently, Brignall, Huczynska and Vatter~\cite{brignall:simple-permutat:} described a framework, called `query-completeness' which concluded that many other subsets of permutations in a class with only finitely many simple permutations also have algebraic generating functions: in the language of this current article, this is done by constructing context-free specifications. 

More recently, a fully algorithmic method to derive combinatorial specifications for permutation classes with only finitely many simple permutations was given by Bassino, Bouvel, Pierrot, Pivoteau, and Rossin~\cite{bassino:an-algorithm-computing:}, and we refer to that article for fuller details.

To begin, we need to adapt the notation used for inflations slightly, to handle the `bottom-to-top' requirement of our combinatorial specifications. Given a permutation $\sigma$ and permutations $\pi_1,\dots,\pi_{|\sigma|}$, the \emph{bottom-to-top inflation} of $\sigma$ by $\pi_1,\dots,\pi_{|\sigma|}$ is the permutation $\sigma[\pi_1,\dots,\pi_{|\sigma|}]^T$ formed by replacing the $\sigma^{-1}(i)$th entry of $\sigma$ by $\pi_i$. That is, the lowest entry of $\sigma$ is inflated by $\pi_1$, the second-lowest by $\pi_2$, and so on. 

In terms of specifications, if $\pi_i\in\C_i$, then the bottom-to-top inflation $\sigma[\pi_1,\dots,\pi_{|\sigma|}]^T$ belongs to the product of classes $\C_1\times \cdots\times \C_{|\sigma|}$. By virtue of the bottom-to-top inflation, this is compatible with bottom-to-top specifications. Furthermore, in order to track the rightmost entry of such an inflation, it is the rightmost entry of the permutation inflating the rightmost entry of $\sigma$, namely $\pi_{\sigma(|\sigma|)}$. Similarly, the leftmost entry of this inflation is the leftmost entry of $\pi_{\sigma(1)}$. 

\begin{proposition}\label{prop:simple-specification}
Every class containing only finitely many simple permutations admits a combinatorial specification that track the leftmost and rightmost entries.
\end{proposition}

\begin{proof}
   Let $\Si(\C)$ denote the (finite) set of simple permutations in $\C$, which for convenience we will assume includes the permutations $12$ and $21$. Consider some combinatorial specification $\rS$ for $\C$, as produced e.g.\ by the algorithm of Bassino et al~\cite{bassino:an-algorithm-computing:}.

   Let $\V$ denote the set of classes and atoms arising in $\rS$, and consider any $\D\in\V$, specified by some equation $\D = f(\V)$. The function $f$ comprises a combinatorial sum of finitely many terms, each of which arises from the inflation of a permutation in $\Si(\C)$ by other classes from $\V$ (or is simply equal to an atom $\Z$). 
   For each such term, we replace the expression with the corresponding expression from the bottom-to-top inflation. Note that this amounts simply to reordering the classes in the term.
   By modifying the ordering of all terms in this way, we obtain a bottom-to-top specification $\rS'$ for $\C$. 
 
 Now, in order to track the rightmost entry, we define $\V^\mathsf{R} = \{\D^\mathsf{R} :\D\in\V\}$ to be the collection of classes that track the rightmost entry.  
 For each $\D^\mathsf{R}\in\V^\mathsf{R}$, we copy the corresponding equation for $\D$ in $\rS'$, except that in each term (corresponding to the inflation of a simple permutation) we identify the rightmost class $\F$, say, and replace it with $\F^\mathsf{R}$. 
 
 A similar process can be carried out to track the leftmost entry as required, as well as tracking both the leftmost and rightmost entries. Combining all the equations thus described for $\V$, $\V^\mathsf{R}$, $\V^\mathsf{L}$ and $\V^\mathsf{LR}$ as required, we obtain a specification for $\CR$ (or $\CL$, $\CLR$) satisfying the requirements of the proposition.
\end{proof}

\paragraph{The class $\Av(321)$}
Permutations of length $n$ in $\Av(321)$ are in bijection with Dyck paths below the diagonal, starting at $(0,0)$ and ending at $(n,n)$ via up ($\sf{U}$) and right ($\sf{R}$) steps. This bijection is well-known, and rather than present a formal definition we merely illustrate the process in Figure~\ref{fig:dyck-321}.

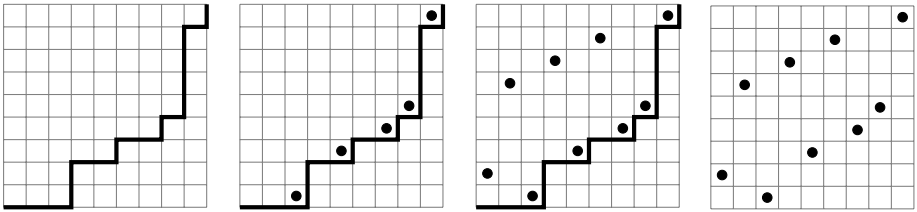
\begin{figure}[!ht]
\begin{center}
\begin{tikzpicture}[scale=0.3]
\draw[gray] (0,0) grid (9,9);
\draw[ultra thick, black] (0,0) -- (3,0) -- (3,2) -- (5,2) -- (5,3) -- (7,3) -- (7,4) -- (8,4) -- (8,8) -- (9,8) -- (9,9);
\end{tikzpicture}\quad 
\begin{tikzpicture}[scale=0.3]
\draw[gray] (0,0) grid (9,9);
\draw[ultra thick, black] (0,0) -- (3,0) -- (3,2) -- (5,2) -- (5,3) -- (7,3) -- (7,4) -- (8,4) -- (8,8) -- (9,8) -- (9,9);
\filldraw[black] (2.5,0.5) circle (6pt);
\filldraw[black] (4.5,2.5) circle (6pt);
\filldraw[black] (6.5,3.5) circle (6pt);
\filldraw[black] (7.5,4.5) circle (6pt);
\filldraw[black] (8.5,8.5) circle (6pt);
\end{tikzpicture}\quad
\begin{tikzpicture}[scale=0.3]
\draw[gray] (0,0) grid (9,9);
\draw[ultra thick, black] (0,0) -- (3,0) -- (3,2) -- (5,2) -- (5,3) -- (7,3) -- (7,4) -- (8,4) -- (8,8) -- (9,8) -- (9,9);
\filldraw[black] (2.5,0.5) circle (6pt);
\filldraw[black] (4.5,2.5) circle (6pt);
\filldraw[black] (6.5,3.5) circle (6pt);
\filldraw[black] (7.5,4.5) circle (6pt);
\filldraw[black] (8.5,8.5) circle (6pt);
\filldraw[black] (5.5,7.5) circle (6pt);
\filldraw[black] (3.5,6.5) circle (6pt);
\filldraw[black] (1.5,5.5) circle (6pt);
\filldraw[black] (0.5,1.5) circle (6pt);
\end{tikzpicture}\quad
\begin{tikzpicture}[scale=0.3]
\fill[white]  (0,0) rectangle +(9,9);
\draw[gray] (0,0) grid (9,9);
\filldraw[black] (2.5,0.5) circle (6pt);
\filldraw[black] (4.5,2.5) circle (6pt);
\filldraw[black] (6.5,3.5) circle (6pt);
\filldraw[black] (7.5,4.5) circle (6pt);
\filldraw[black] (8.5,8.5) circle (6pt);
\filldraw[black] (5.5,7.5) circle (6pt);
\filldraw[black] (3.5,6.5) circle (6pt);
\filldraw[black] (1.5,5.5) circle (6pt);
\filldraw[black] (0.5,1.5) circle (6pt);
\end{tikzpicture}
\end{center}
\caption{Moving from a Dyck path to its corresponding 321-avoider.}
\label{fig:dyck-321}
\end{figure}

In this bijection, the rightmost entry of a 321-avoider is the first right-to-left minimum, which occurs in the final `corner', i.e.\ the first $\sf{U}$ step after the last $\sf{R}$ step.\footnote{It is also possible to identify the leftmost entry, though for brevity we have not done this here.} Thus, the class of Dyck paths $\D$ satisfies the following combinatorial specification (taking the atoms to be $\sf{R}$, $\sf{U}$, and (to mark the `rightmost corner') $\sf{U^R}$.
\begin{align*}
\D^{\sf{R}} &= \D\sf{R}\D^{\sf{R}}\sf{U}+\D \sf{RU^R}\\
\D &= \EE + \D\sf{R}\D\sf{U}
\end{align*}

Translating this to the permutation case (by mapping $\sf{U}$ to $\Z$, $\sf{U^R}$ to $\ZR$, and ignoring $\sf{R}$) yields the following specification:
\begin{align*}
\CR &= \C\CR\Z + \C\ZR\\
\C &= \EE + \C\C\Z.
\end{align*}

\paragraph{Insertion encodings}

Introduced by Albert, Linton and Ru\v{s}kuc~\cite{albert:the-insertion-e:}, the \emph{insertion encoding} of a class is a method of constructing permutations in the class by repeatedly adding a new maximal element into a number of permitted active sites called \emph{slots}. The addition of a new maximal element can be performed in one of four ways: fill the slot ($\mathsf{F}$), insert to the right of a slot ($\mathsf{R}$, leaving a slot to the left of the newly-inserted entry), insert to the left of a slot ($\mathsf{L}$), or insert in the middle of a slot ($\mathsf{M}$, leaving a slot on both sides of the newly-inserted entry). To these four basic letters, one typically uses indices to denote which slot (reading, e.g., from left to right) the action is performed on. For example, the permutation 3164752 is encoded by $\mathsf{M_1R_2F_1M_1R_2F_1F_1}$, illustrated as follows (Reading upwards):
\begin{align*}
3 1 6 &4 7 5 2 \\
3 1 6 &4 \diamond 5 2 \\
3 1 \diamond &4 \diamond 5 2 \\
3 1 \diamond &4 \diamond 2 \\
3 1& \diamond 2 \\
\diamond 1& \diamond 2 \\
\diamond &1 \diamond  \\
&\diamond
\end{align*}
By their nature, insertion encodings for permutation classes are bottom-to-top processes. Thus, in order to include classes which possess specifications given by insertion encodings amongst those to whom our method can be applied,\footnote{Note that, for a class $\C$ which possesses a regular or context-free insertion encoding, the use of our method may be overkill: for example, to count $\C\mid\Av(21)$, it may be easier simply to adjust the insertion encoding by permitting an additional slot on the right (subject to ensuring each permutation is represented uniquely, of course).} it suffices to capture how to track the rightmost (and/or the leftmost) entry. 

We make the following observation.

\begin{lemma}
In any insertion encoding of a permutation class, the rightmost entries of permutations in the class correspond to an $\mathsf{F}$ or $\mathsf{R}$ insertion into the rightmost gap.
\end{lemma}

For example, the following grammar (given in Albert,Linton, Ru\v{s}kuc~\cite{albert:the-insertion-e:}) defines the insertion encoding of $\Av(312)$. (Note that, as there is only one active slot, we have omitted all indices.)
\[\mathsf{s} \rightarrow \mathsf{F} \mid \mathsf{Ls}\mid \mathsf{Rs}\mid \mathsf{Mss}.\]
To adapt this encoding to track the rightmost entry, we introduce two new symbols, $\mathsf{F^R}$ and $\mathsf{R^R}$, corresponding to the insertion of the rightmost entry: thus, any word in this encoding has exactly one of these two symbols, occurring exactly once. The grammar becomes:
\begin{align*}
\mathsf{s^R} &\rightarrow \mathsf{F^R} \mid \mathsf{Ls}\mid \mathsf{R^Rs}\mid \mathsf{Mss^R}\\
\mathsf{s} &\rightarrow \mathsf{F} \mid \mathsf{Ls}\mid \mathsf{Rs}\mid \mathsf{Mss}
\end{align*}
where $\mathsf{s^R}$ denotes the encoding for the class with the rightmost entry encoded separately. Converting to combinatorial specifications for the class $\Av(312)$, we obtain
\begin{align*}
\CR &= \ZR+\Z\CR+\ZR\C+\Z\C\CR\\
\C &= \Z + \Z\C	+\Z\C + \Z\C\C.
\end{align*}
This specification can be simplified slightly by admitting the empty permutation into $\C$:
\begin{align*}
\CR &= \ZR\C+\Z\C\CR\\
\C &= \EE + \Z\C\C.
\end{align*}

Similar modifications can be made to track the leftmost entry; we omit the details.
%
%
%
%
%
%
%
%
\section{Examples}\label{sec-examples}

In this section, we give three examples of the application of our operators to derive combinatorial specifications (and hence enumerations) for juxtapositions.

\subsection{\texorpdfstring{$\Av(321)\mid \Av(21)$}{Av(321)|Av(21)}}

Juxtaposing the 321-avoiding permutations with a monotone increasing sequence was first carried out by the current authors~\cite{bs:juxt-catalan}. We repeat that enumeration here using the rightmost-entry-tracking specification obtained in the previous section:
\begin{align*}
\CR &= \C\CR\Z + \C\ZR\\
\C &= \EE + \C\C\Z.
\end{align*}
From the specification above, together with the master equation,
\[\CR\mid\Av(21)=\EE+\SZ\times\ZR+\opi(\CR)+\opii(\CR)+\opio(\CR)\times \SZ\times\ZR.\]
we obtain expansions for the classes in $\V_\bop$, where $\V=\{\C,\CR\}$. To aid readability, we use the shorthand $\CR_\i = \opi(\CR)$, $\C_\ii= \opii(\C)$ and so on.
\begin{align*}
\CR_\i &= \C_\i\CR_\o\Z + \C_\o\CR_\i\Z + \C_\i\Z + \C_\o\ZR\Z\\
\CR_\ii &= \C_\ii(\CR_\o\Z + \Z) +\C_\io(\CR_\oi\Z +\CR_\oo\SZ\ZR\Z + \SZ\ZR\Z) +\C_\o(\CR_\ii\Z + \CR_\io\SZ\ZR\Z+  \Z\SZ\ZR\Z)\\
\CR_\oi &= \C_\oi(\CR_\o\Z + \Z)+ \C_\oo(\CR_\oi\Z + \CR_\oo\SZ\ZR\Z +  \SZ\ZR\Z)\\
\CR_\io &= \C_\io(\CR_\oo\SZ\Z + \SZ\Z) + \C_\o(\CR_\io\SZ\Z + \Z\SZ\Z)\\
\CR_\oo &= \C_\oo(\CR_\oo\SZ\Z + \SZ\Z)\\
\CR_\o &= \C_\o(\CR_\o\Z + \Z)\\
\C_\i &= \C_\i\C_\o\Z + \C_\o\C_\i\Z + \C_\o\C_\o\ZR\Z\\
\C_\ii &= \C_\ii\C_\o\Z + \C_\io(\C_\oi\Z + \C_\oo\SZ\ZR\Z) + \C_\o(\C_\ii\Z + \C_\io\SZ\ZR\Z + \C_\o\Z\SZ\ZR\Z)\\
\C_\oi &= \C_\oi\C_\o\Z + \C_\oo(\C_\oi\Z + \C_\oo\SZ\ZR\Z)\\
\C_\io &= \C_\io\C_\oo\SZ\Z + \C_\o(\C_\io\SZ\Z + \C_\o\Z\SZ\Z)\\
\C_\oo &= \EE + \C_\oo\C_\oo\SZ\Z\\
\C_\o &= \EE + \C_\o\C_\o\Z_\o.
\end{align*}

Solving, simplifying and converting into generating functions yields the following,
\newcommand{\catsqrt}{\sqrt{1-4z}}
\[ f(z)  = -\frac{1-\catsqrt + z(-4+\catsqrt + \sqrt{1-5z}/\sqrt{1-z})}{2z^2}.\]
which corresponds to OEIS sequence A278301~\cite{oeis}.

\subsection{\texorpdfstring{$\Av(21)\mid\Av(21)\mid\Av(21)$}{Av(21)|Av(21)|Av(21)}}
In this example, we will demonstrate the use of the operators acting on sequences, together with the use of the $\Phi$ and $\Theta$ operators. For this, our construction of the double juxtaposition $\Av(21)\mid\Av(21)\mid\Av(21)$ will be somewhat contrived, appending one class $\Av(21)$ to the right, then appending the other increasing sequence to the left (via symmetries).

We begin with the following regular specification for $\Av(21)$ that tracks both the leftmost and rightmost entries:
\[\C = \ZLR + \ZL\seq(\Z)\ZR.\]
Since we need to keep track of the leftmost entry, we use the following equation for our first juxtaposition.
\[\C\mid\Av(21) =\EE+\ZLR +\ZL\seq(\Z)\ZR+\opi(\C)+\opii(\C)+\opio(\C)\seq(\Z)\ZR.\]
However, as we will only be appending to each side once, we do not need to keep track of the rightmost entry in the above expression. We can therefore use the following simplified version:
\begin{equation}\label{eq:simplified}
	\C\mid\Av(21) =\EE+\ZL\seq(\Z) +\opio(\C)\seq(\Z).
\end{equation}
The only term we need to expand in the above is $\opio(\C)$, and we do this first, separately. For ease of viewing, we use $\SZ$ for $\seq(\Z)$, and make use of expressions such as $\EE + \SZ\Z = \SZ$ to simplify the presentation. Similarly, we will also use the superscript $+$ to denote non-empty sequences, thus we have, for example, $\SZp = \SZ\Z = \Z\SZ$ and $\seqp(\A) = \seq(\A)\A$.
\[
\opio(\C) = \left(\SZp\ZL  + \ZL\SZp\SZp\right) \seq(\SZp).
\]
Thus
\[
\C\mid\Av(21) =\EE+\ZL\SZ +\left(\SZp\ZL  + \ZL\SZp\SZp\right) \seq(\SZp)\SZ.
\]

Momentarily evaluating this as a generating function, we have
\[f_{\C\mid\Av(21)}(z) = \frac{1}{1-z} + \frac{z^2}{(1-z)^2(1-2z)}\]
as expected.

We now turn to appending $\Av(21)$ to the left of $\C\mid\Av(21)$. First, we need to apply the operators $\Phi$ and $\Theta$ to the specification above, yielding the following specification for $\Av(21)\mid\C^{rc}$:
\begin{align*}
\Av(21)\mid\C^{rc} &=\EE+\SZ\ZR + \SZ \seq(\SZp) \left(\ZR\SZp  + \SZp\SZp \ZR\right)\\
&=\EE+\SZ\ZR +  \seq(\SZp)\left(\SZ\ZR\SZp  +  \SZp\SZp \SZ\ZR\right)
\end{align*}
Again, we will use the simplified equation~(\ref{eq:simplified}) to find the specification for $\Av(21)\mid\C^{rc}\mid\Av(21)$, without keeping track of the rightmost entry. This requires us to compute $\opio(\Av(21)\mid\C^{rc})$:
\begin{align*}
\opio(\Av(21)\mid\C^{rc}) &= \SZp\seqp(\SZp) + \seqp(\SZp)\seqp(\SZp)\seqp(\seqp(\SZp))\seqp(\SZp)\seq(\SZp)  \\
&\qquad + \seqp(\SZp) \left[  \seq(\SZp)\seqp(\SZp) + \SZp\seqp(\SZp) + \SZp\SZp  \right]\seqp(\SZp).
\end{align*}
Inserting this into equation~(\ref{eq:simplified}) yields a specification for $\Av(21)\mid\C^{rc}\mid\Av(21)$. Moving to generating functions, and noting that the classes $\Av(21)\mid\C^{rc}\mid\Av(21)$ and $\Av(21)\mid\C\mid\Av(21)$ are identical (and hence equinumerous), we obtain the following generating function. 
\[
f_{\Av(21)\mid\Av(21)\mid\Av(21)} = \frac{1}{1-z} + \frac{z^2}{(1-z)^2(1-2z)} + \frac{z^3(1+z-4z^2)}{(1-z)^3(1-2z)^2(1-3z)}.
\]
This agrees with the expression given in Bevan's thesis~\cite[page~34]{bevan:phd}, and gives sequence A326355 in the OEIS~\cite{oeis}.

\subsection{Separable permutations}

  The following is an illustrated bottom-to-top specification for the class of \emph{non-empty} separable permutations, $\Av(2413,3142)$:
  \begin{align*}
    \C &= \Z + \cplusc{\Cp}{\C} + \cminusc{\C}{\Cm}\\
    \Cm &= \Z + \cplusc{\Cp}{\C}\\
    \Cp &=\Z + \cminusc{\C}{\Cm}
  \end{align*}
  where $\Cm$ and $\Cp$ denote the skew-indecomposable and sum-indecomposable permutations in $\C$, respectively. The following specification then tracks the rightmost entry:
  \begin{align*}
    \CR &= \ZR + \Cp\CR + \CR\Cm\\
    \C &= \Z + \Cp\C + \C\Cm\\
    \Cm &= \Z + \Cp\C\\
    \Cp &= \Z + \C\Cm.
  \end{align*}
Note that there are other specifications that carry out the same process, but this one has been chosen to minimise the number of combinatorial classes required in the final specification for the juxtaposition we will compute below, namely the juxtaposition of the separable permutations with the class $\Av(21)$. 

First, although we do not need to know the basis of this class for our method, we can compute that $\C\mid\Av(21)$ is equivalent to
\[\Av(25143,35142,35241,41532,42531,241365,251364,314265,315264,415263).\]
  
To describe the complete specification for $\CR\mid\Av(21)$ that tracks the rightmost entry requires 26 separate equations. Instead, here we will again drop the rightmost tracking in order to compute the generating function for $\CR\mid\Av(21)$ using the simplified equation~(\ref{eq:simplified}). This requires the following 11 equations (together with the equation for $\SZ$, and the equation for the class itself).
\begin{align*}
\CR_\io &= \Z\SZ\Z + \Cp_\io\CR_\oo + \Cp_\o\CR_\io + \CR_\io \Cm_\oo\\
\CR_\oo &= \SZ\Z +\Cp_\oo\CR_\oo + \CR_\oo\Cm_\oo\\
\C_\io &= \Z\SZ\Z + \Cp_\io\C_\oo + \Cp_\o\C_\io + \C_\io\Cm_\oo + \C_\o\Cm_\io\\
\C_\oo &= \SZ\Z + \Cp_\oo\C_\oo + \C_\oo\Cm_\oo\\
\C_\o &= \Z+\Cp_\o\C_\o + \C_\o\Cm_\o\\
\Cm_\io &= \Z\SZ\Z + \Cp_\io\C_\oo + \Cp_\o\C_\io\\
\Cm_\oo &= \SZ\Z + \Cp_\oo\C_\oo \\
\Cm_\o &= \Z+\Cp_\o\C_\o\\
\Cp_\io &= \Z\SZ\Z + \C_\io\Cm_\oo + \C_\o\Cm_\io\\
\Cp_\oo &=\SZ\Z + \C_\oo\Cm_\oo \\
\Cp_\o &= \Z + \C_\o\Cm_\o
\end{align*}
Solving, we obtain the following
\begin{multline*}
	f_{\C\mid\Av(21)}(z) = \frac{(2-4z+z^2)xy}{4(1-z)(-2+7z-7z^2+z^3)} +\\ \frac{(-2+10z-15z^2+7z^3)x + (2-6z+z^2+6z^3-z^4)y -10+54z-99z^2+66z^3-9z^4}{4(1-z)^2(-2+7z-7z^2+z^3)}
\end{multline*}
where 
\[ x=\sqrt{1-6z+z^2} \quad\text{and}\quad y=\sqrt{1-8z+8z^2}.\]
The first ten terms of the counting sequence are
\[ 1,1,2,6,24,115,609,3409,19728,116692,701062.\] 
This is sequence A326348 in OEIS~\cite{oeis}.
%
%
%
%
%
%
%
%
\section{Concluding remarks}\label{sec-conclusion}

\paragraph{Acyclic grids with one non-monotone}
A natural next step beyond the $k\times 1$ grids considered here are families of \emph{acyclic} grids in which at most one cell is non-monotone. This poses a number of challenges beyond the methods described in this article, notably to incorporate the ability to switch from a bottom-to-top specification to a left-to-right one, depending on where a new monotone cell needs to be appended. Another challenge is to identify some canonical way to handle griddings that is consistent in both directions.

\paragraph{Multiple non-monotone}
One can likely replace the monotone classes in our framework with certain other easy-to-describe classes, such as the class $\Av(312,231)$ of layered permutations.  This adds considerably to the number and complexity of the operators required, but the same principles apply.

On the other hand, our approach does not seem to offer insight for juxtapositions involving more complicated classes, for example juxtaposing two classes with algebraic generating functions. A notable example of such a juxtaposition is $\Av(132)\mid\Av(213)$, which is the superclass of the ``domino'' used in recent bounds on the growth rate of $\Av(1324)$~\cite{bbep:1324}.

\bibliographystyle{acm}
\bibliography{../refs}

\end{document}

%% file: tikzperm.tex
\usepackage{ifthen}
\usetikzlibrary{calc}

\newcommand{\plotptradius}{4pt}

\tikzset{permpt/.style={circle, draw, fill=black, inner sep=0pt, minimum width=\plotptradius}}
\tikzset{empty/.style={draw=none, fill=none}}

\newcommand\absdot[2]{
	\node[permpt] at #1 {};
}

\newcommand{\plotperm}[2][black]{ 
	\foreach \j [count=\i] in {#2} {
        \ifnum0=\j {} \else {
 		\node[permpt,fill=#1,draw=#1] (\j) at (\i,\j) {};
	} \fi
	};
}

\newcommand{\plotpinsequence}[1]{
	\absdot{(0,0)}{};
	\edef\n{0}
	\edef\s{0}
	\edef\e{0}
	\edef\w{0}
	\edef\x{0}
	\edef\y{0}
	\foreach \pin [remember=\pin as \oldpin (initially 1), count=\i] in {#1} {
		\ifthenelse{\pin=1 \OR \pin=2}{
			\ifthenelse{\oldpin=3}{
				\xdef\x{\number\numexpr\e-1}
			}{
				\xdef\x{\number\numexpr\w+1}
			}
			\ifnum\i=1 
				\pgfmathparse{\e+1}
 				\xdef\e{\pgfmathresult}
			\fi	
		}{ 
			\ifthenelse{\oldpin=1}{
				\xdef\y{\number\numexpr\n-1}
			}{
				\xdef\y{\number\numexpr\s+1}
			}
			\ifnum\i=1 
				\pgfmathparse{\s-1}
 				\xdef\s{\pgfmathresult}
			\fi	
		}
		\ifnum\pin=1 
			\pgfmathparse{\n+2}
 			\xdef\n{\pgfmathresult}		
			\absdot{(\x,\n)}{};
			\ifnum\i>1
				\draw (\x,\n) -- (\x,\y-0.5);
			\else
				\draw[gray,very thick] (-0.5,-0.5) rectangle (\x+0.5,\n+0.5);
			\fi
		\fi
		\ifnum\pin=2 
			\pgfmathparse{\s-2}
 			\xdef\s{\pgfmathresult}
			\absdot{(\x,\s)}{};
			\ifnum\i>1
				\draw (\x,\s) -- (\x,\y+0.5);
			\else
				\draw[gray,very thick] (-0.5,0.5) rectangle (\x+0.5,\s-0.5);
			\fi
		\fi
		\ifnum\pin=3 
			\pgfmathparse{\e+2}
 			\xdef\e{\pgfmathresult}
			\absdot{(\e,\y)}{};
			\ifnum\i>1
				\draw (\e,\y) -- (\x-0.5,\y);
			\else
				\draw[gray,very thick] (-0.5,+0.5) rectangle (\e+0.5,\y-0.5);
			\fi
		\fi
		\ifnum\pin=4 
			\pgfmathparse{\w-2}
 			\xdef\w{\pgfmathresult}
			\absdot{(\w,\y)}{};
			\ifnum\i>1
				\draw (\w,\y) -- (\x+0.5,\y);
			\else
				\draw[gray,very thick] (0.5,0.5) rectangle (\w-0.5,\y-0.5);

			\fi
		\fi		
	};
}